\documentclass[12pt]{amsart}
\pdfoutput=1

\usepackage{amssymb,amsfonts}
 \usepackage[foot]{amsaddr}
\usepackage{amsmath}
\usepackage{amsthm}
\usepackage{enumerate}
\usepackage{thmtools}
\usepackage{datetime}
\usepackage{graphicx}
\usepackage{appendix}
\usepackage{tikz}
\usetikzlibrary{patterns}
\usetikzlibrary{arrows.meta}
\usetikzlibrary{arrows}
\usepackage{bbm}
\definecolor{darkspringgreen}{rgb}{0.14, 0.7, 0.3}
\definecolor{melon}{rgb}{0.4, 0.2, 1}
\usepackage[colorlinks, citecolor = darkspringgreen, linkcolor = melon]{hyperref}

\calclayout

\setboolean{@twoside}{false}

\newtheorem{thm}{Theorem}[section]
\newtheorem{cor}[thm]{Corollary}
\newtheorem{prop}[thm]{Proposition}
\newtheorem{lem}[thm]{Lemma}
\newtheorem{conj}[thm]{Conjecture}

\theoremstyle{definition}
\newtheorem{defn}[thm]{Definition}

\newtheorem{exmp}[thm]{Example}

\theoremstyle{remark}
\newtheorem{rem}[thm]{Remark}

\newtheorem{obs}[thm]{Observation}

\makeatletter
\let\c@equation\c@thm
\makeatother
\numberwithin{equation}{section}

\makeatletter
\newcommand*\bigcdot{\mathpalette\bigcdot@{.5}}
\newcommand*\bigcdot@[2]{\mathbin{\vcenter{\hbox{\scalebox{#2}{$\m@th#1\bullet$}}}}}
\makeatother

\makeatletter
\def\subsection{\@startsection{subsection}{3}%
  \z@{.5\linespacing\@plus.7\linespacing}{.1\linespacing}%
  {\bfseries}}
\makeatother

\makeatletter
\def\@tocline#1#2#3#4#5#6#7{\relax
  \ifnum #1>\c@tocdepth 
  \else
    \par \addpenalty\@secpenalty\addvspace{#2}%
    \begingroup \hyphenpenalty\@M
    \@ifempty{#4}{%
      \@tempdima\csname r@tocindent\number#1\endcsname\relax
    }{%
      \@tempdima#4\relax
    }%
    \parindent\z@ \leftskip#3\relax \advance\leftskip\@tempdima\relax
    \rightskip\@pnumwidth plus4em \parfillskip-\@pnumwidth
    #5\leavevmode\hskip-\@tempdima
      \ifcase #1
       \or\or \hskip 1em \or \hskip 2em \else \hskip 3em \fi%
      #6\nobreak\relax
    \hfill\hbox to\@pnumwidth{\@tocpagenum{#7}}\par
    \nobreak
    \endgroup
  \fi}
\makeatother

\newlength{\cellsize}
\newcommand\tableau[1]{
\vcenter{
\let\\=\cr
\baselineskip=-16000pt
\lineskiplimit=16000pt
\lineskip=0pt
\halign{&\tableaucell{##}\cr#1\crcr}}}

\newcommand{\tableaucell}[1]{{%
\def \arg{#1}\def \void{}%
\ifx \void \arg
\vbox to \cellsize{\vfil \hrule width \cellsize height 0pt}%
\else
\unitlength=\cellsize
\begin{picture}(1,1)
\put(0,0){\makebox(1,1)[c]{$#1$}}
\put(0,0){\line(1,0){1}}
\put(0,1){\line(1,0){1}}
\put(0,0){\line(0,1){1}}
\put(1,0){\line(0,1){1}}
\end{picture}%
\fi}}

\newcommand{\C}{\mathbb{C}}

\newcommand{\R}{\mathbb{R}}
\newcommand{\Z}{\mathbb{Z}}

\newcommand{\calH}{\mathcal{H}}

\newcommand{\Th}{\text{th}}
\DeclareMathOperator{\Kostka}{\sf Kostka}
\DeclareMathOperator{\Par}{{\sf Par}}

\DeclareMathOperator{\GL}{GL}

\DeclareMathOperator{\cut} {\setminus}
\DeclareMathOperator{\cyclic} {\Delta}
\newcommand{\Kpoly}{\text{P}^{\text{\sf Kostka}}}

\title{On Faces and Hilbert Bases of Kostka Cones}

\author[Burcroff]{Amanda Burcroff}
\address{Department of Mathematics, Harvard University, Cambridge, MA}
\email{aburcroff@math.harvard.edu}

\begin{document}

\begin{abstract}
Kostka coefficients appear in the representation theory of the general linear group and enumerate semistandard Young tableaux of fixed shape and content.  The $r$-Kostka cone is the real polyhedral cone generated by pairs of partitions with at most $r$ parts, written as non-increasing $r$-tuples, such that the corresponding Kostka coefficient is nonzero.  We provide several results showing that its faces have interesting structural and enumerative properties.  We show that the $d$-faces of the $r$-Kostka cone can be determined from those of the $(3d+1)$-Kostka cone, allowing us to characterize its $2$-faces and enumerate its $d$-faces for  $d \leq 4$.  We provide tight asymptotics for the number of $d$-faces for arbitrary $d$ and determine the maximum number of extremal rays contained in a $d$-face for $d < r$.  We then make progress towards a generalization of the Gao-Kiers-Orelowitz-Yong Width Bound on initial entries of partitions $(\lambda,\mu)$ appearing in the Hilbert basis of the $\lambda_1$-Kostka cone.  We show that at least $93.7\%$ of integer pairs $\lambda_1 \geq \mu_1 > 0$ appear as the initial entries of partitions $(\lambda,\mu)$ comprising a Hilbert basis element of the $r$-Kostka cone for every $r > \lambda_1$.  We conclude with a conjecture about a curious $h$-vector phenomenon.
\end{abstract}

\maketitle
\tableofcontents

\addtocontents{toc}{\protect\setcounter{tocdepth}{1}}

\section{Introduction}\label{sec: introduction}
\subsection{Background}
The $r$-Kostka cone, denoted by $\Kostka_r$, is the real polyhedral cone generated by pairs $(\lambda,\mu) \in \R^{2r}$ of non-increasing $r$-tuples of equal sum such that, for all $1 \leq i < r$, the sum of the first $i$ parts of $\lambda$ is at least the sum of the first $i$ parts of $\mu$.  It is directly connected to the well-known \emph{Kostka numbers}, which in turn have connections to Young tableaux \cite{littlewood1938construction}, representation theory \cite{fulton2013representation}, symmetric functions \cite{kostka1882ueber}, dimer configurations \cite{korff2017dimers}, and supergravity theories \cite{taylor2019generic}.

The integral points of the $r$-Kostka cone are precisely the pairs $(\lambda,\mu)$ of integer partitions with at most $r$ parts such that the Kostka number $K_{\lambda,\mu}$ is positive.  Carl Kostka introduced Kostka numbers in 1882 while studying symmetric function expansions \cite{kostka1882ueber}.  Kostka numbers are hard to compute in general, as their computation is ${\sf \# P}$-complete \cite{narayanan2006complexity}.  Kostka numbers also appear in the representation theory of the general linear group. By Young's Rule, the Kostka number $K_{\lambda,\mu}$ is the multiplicity with which the weight $\mu$ appears in the irreducible representation of $\GL_r(\C)$ with highest weight $\lambda$.  It is also the coefficient of the monomial symmetric function corresponding to $\mu$ in the expansion of the Schur polynomial corresponding to $\lambda$.  See, \cite[Chapter 7]{Stanley_EC2_1999} for a more thorough history of Kostka numbers and \cite{fulton2013representation} for details on the representation-theoretic perspective.

Slicing the $r$-Kostka cone by the affine hyperplane $\{x \in \R^{2r}: (1,1,\dots,1) \cdot x = 1\}$ yields a $(2r-2)$-dimensional polytope, which we call the \emph{Kostka polytope} and denote by $\Kpoly_r$.  There are numerous other polytopes defined in terms of partitions, the faces of which have previously been shown to have interesting enumerative properties.  The \emph{Fibonacci polytopes}, or \emph{ordered partition polytopes},  have vertex sets satisfying a Fibonacci-like recurrence \cite{rispolifibonacci} and are related to alternating permutations \cite{stanley2010survey}.  For the family of \emph{unordered partition polytopes}, Shlyk gave a description of the dynamic behavior of the vertices and a characterization of the facets \cite{SHLYK20051139}.  Each unordered partition polytope is combinatorially equivalent to a face of $\Kpoly_r$, and computational evidence suggests that both polytope families share a curious $h$-vector phenomenon \cite{wangnumber} (see \autoref{sec: further directions}).  

Several recent works on the Kostka cone have focused on its Hilbert basis and extremal rays. In 2021, Gao, Kiers, Orelowitz, and Yong \cite{gao2021kostka} gave a criterion for Hilbert basis membership, though they show that this decision problem is ${\sf NP}$-complete in general.  They use this criterion to give a simple description of the extremal rays and a ``Width Bound'' on the integer pairs $(\lambda_1,\mu_1)$ that can be the first parts of partitions $\lambda,\mu$ forming a Hilbert basis element $(\lambda,\mu)$ of the $r$-Kostka cone for $r \leq \lambda_1$.   Kim has since provided a strengthening of this Width Bound via a study of generalized Dyck paths \cite{kim2021kostka}.  Similar studies have also been carried out in other Lie types.  Besson, Jeralds, and Kiers \cite{besson2021vertices} took a representation-theoretic approach to enumerate the rays of the \emph{generalized Kostka cones} of types $D_r$ and $E_r$, where type $A_r$ is the classical case handled in \cite{gao2021kostka}.

\subsection{Results}

Our work focuses on studying the faces and Hilbert basis of the $r$-Kostka cone $\Kostka_r$, with a focus on enumerative and structural properties. We typically refer to $r$-Kostka polytope $\Kpoly_r$ instead of the Kostka cone when discussing the face structure, as $d$-faces of $\Kpoly_r$ are naturally identified with $(d+1)$-faces of $\Kostka_r$.  We begin by studying the maximum number of vertices contained in a face of fixed dimension (see \autoref{cor: max face size}).  

\begin{thm}\label{thm: max face size}
For $r > d + 1$, the maximum number of vertices contained in a $d$-face of the polytope $\Kpoly_r$ is 
$\displaystyle\prod_{i=1}^3 \left\lfloor \frac{d+2+i}{3} \right \rfloor$,
which is the maximum product of three positive integers summing to $d+3$.
\end{thm}

We then characterize the edges of $\Kpoly_r$ using a connection to cells of the braid arrangement.  As is explained in \autoref{sec: prelims}, the vertices of $\Kpoly_r$ can be labeled by integer triples, and the edge characterization is given in terms of certain inequalities on the vertex labels (\autoref{thm: 2d faces}).  By reducing the $d$-face structure of $\Kpoly_r$ to that of $\Kpoly_{3d+3}$ (\autoref{thm: face recursion formula}), we can provide exact formulas for the number of $d$-faces of $\Kpoly_r$ for $d = 1,2,3$.  
\begin{thm}\label{thm: face numbers}
        The number of edges of $\Kpoly_r$ is
        $$f_1(r) = \binom{r}{6} + 2\binom{r}{5} + 6\binom{r}{4} + 7\binom{r}{3} + 3\binom{r}{2}\,,$$
        the number of two-dimensional faces of $\Kpoly_r$ is
\begin{align*}
f_2(r) = \binom{r}{9} + 3\binom{r}{8} + 12\binom{r}{7} + 23\binom{r}{6} + 33\binom{r}{5} + 31\binom{r}{4} + 13\binom{r}{3} + \binom{r}{2}\,,\end{align*}
and the number of three-dimensional faces of $\Kpoly_r$ is
\begin{align*}
f_3(r) = \binom{r}{12} &+ 4\binom{r}{11} + 19\binom{r}{10} + 49\binom{r}{9} + 105\binom{r}{8} + 163\binom{r}{7} + 177\binom{r}{6}\\
&+ 131\binom{r}{5} + 53\binom{r}{4} + 7\binom{r}{3}\,.\end{align*}
\end{thm}
These face counting functions have positive integer coefficients in terms of the polynomial basis $\binom{r}{k}_{k \geq 0}$, and we show that this property holds in all dimensions.  We also determine that the coefficient of the top degree term $\binom{r}{3d+3}$ is always $1$, yielding precise asymptotics for the number of $d$-faces. 

The main result of the last section concerns the Hilbert basis of $\Kostka_r$.  We say that an integer pair $(\lambda_1,\mu_1)$ is \emph{$r$-initial} if there is an element $(\lambda,\mu)$ in the Hilbert basis of $\Kostka_r$ such that $\lambda$ has first element $\lambda_1$ and $\mu$ has first element $\mu_1$.  The Width Bound of Gao-Kiers-Orelowitz-Yong \cite[Theorem 1.4]{gao2021kostka} implies that $(\lambda_1,\mu_1)$ is $\lambda_1$-initial if and only if $\lambda_1$ and $\mu_1$ are coprime.  We provide several sufficient conditions for a pair $(\lambda_1,\mu_1)$ to be $(\lambda_1 + 1)$-initial, and these conditions hold for over $93.7\%$ of integer pairs $\lambda_1 \geq \mu_1$. 

\begin{thm}\label{thm: initial main}
If any of the following conditions hold:
\begin{itemize}
\item $\lambda_1$ and $\mu_1$ are coprime \cite[Theorem 1.4]{gao2021kostka}, or
\item $\lambda_1 + 1$ and $\mu_1$ are coprime, or
\item $\lambda_1 + 1$ and $\mu_1 + 1$ are coprime with $2\mu_1 \geq \lambda_1$,
\end{itemize}
then the pair $(\lambda_1, \mu_1)$ is $(\lambda_1 + 1)$-initial.  Moreover, this holds even if we consider only Hilbert basis elements on the $2$-faces of $\Kostka_r$.
\end{thm}
The first criterion follows directly from the work of Gao-Kiers-Orelowitz-Yong, while the latter two conditions are the result of new constructions of Hilbert basis elements.  We conclude with a new observation that, for small $r$, half of the $h$-vector entries for $\Kostka_r$ are $1$, and we conjecture that this holds in general.

\subsection{Outline}
We begin by providing some preliminaries on the Kostka cone and Kostka polytope in \autoref{sec: prelims}.  We study the maximum number of vertices contained in a face of the Kostka polytope in \autoref{sec: max face}.  The edge characterization of the Kostka polytope is in \autoref{sec: edge description}, and the enumerative results on the faces of fixed dimension are in  \autoref{sec: enumeration}.  The construction of Hilbert basis elements is discussed in \autoref{sec: hilbert}, with some computation relegated to the \hyperref[app: initial pair]{Appendix}.  We conclude with a discussion of further directions in \autoref{sec: further directions}.

\addtocontents{toc}{\protect\setcounter{tocdepth}{0}}

\section*{Acknowledgements}
This work was completed in part at the 2022 Graduate Research Workshop in Combinatorics, which was supported in part by NSF grant $\#$1953985, and a generous award from the Combinatorics Foundation.  This paper is the result of many fruitful discussions with Shiliang Gao and Sheila Sundaram.  The author deeply thanks Shiliang Gao for suggesting this topic at the 2022 GRWC and for his helpful contributions throughout the course of the project.  She is extremely grateful to Sheila Sundaram, as this project would not have been possible without her insight and generous support. The author also thanks Margaret Bayer, Yibo Gao, Jeremy Martin, and Tyrrell McAllister for their early contributions to this project.  She appreciates the comments of Richard Stanley and Charles Wang on the $h$-vector phenomenon discussed in the Further Directions section. The author extends her thanks to Niven Achenjang for helping to compute the probability in \autoref{cor: initial prob} and to Joshua Kiers for sharing the code used to discover \autoref{exmp: initial counterexample}. 

\addtocontents{toc}{\protect\setcounter{tocdepth}{1}}

\section{Preliminaries}\label{sec: prelims}
\subsection{The Kostka Cone}
For positive integers $r$ and $n$, we denote the set of integer partitions of $n$ into at most $r$ parts by ${\Par}_r(n)$, where such partitions are written as non-increasing $r$-tuples. Each partition can be viewed as a Young diagram, where the length of the $i^{\Th}$ row is the $i^{\Th}$ entry of the $r$-tuple.  

Consider two partitions $\lambda = (\lambda_1,\dots,\lambda_r)$ and $\mu = (\mu_1,\dots,\mu_r)$ in $\Par_r(n)$.  A \emph{semistandard tableau of shape $\lambda$ and content $\mu$} is a filling of the Young diagram corresponding to $\lambda$ with integer entries such that the rows are non-decreasing to the right, the columns strictly increase downward, and there are precisely $\mu_i$ boxes with entry $i$ for all $1 \leq i \leq r$.  These are counted by the \emph{Kostka coefficient} $K_{\lambda,\mu}$. 

\begin{exmp}
The Kostka coefficient $K_{(4,2),(2,2,1,1)}$ is equal to $4$, as shown by the following four tableaux of shape $(4,2)$ and content $(2,2,1,1)$.
$$\tableau{1  & 1 &2 &2 \\ 3 &4 } \qquad \tableau{1  &1  &2  &3 \\ 2 &4 }\qquad \tableau{1  &1 &2 &4 \\ 2 &3 }\qquad \tableau{1  &1 &3 &4 \\ 2 &2 }$$
\end{exmp}

There is a well-known condition for when a Kostka coefficient is nonzero.  This occurs precisely when $\lambda$ \emph{dominates} $\mu$, i.e.,
$$\sum_{i=1}^k \lambda_i \geq \sum_{j=1}^k \mu_j \text{ for all } k \leq r\,.$$
This is denoted by $\lambda\geq_{\sf Dom} \mu $, and this ordering on partitions is called the \emph{dominance order} (also known as the \emph{majorization order} or \emph{natural order}) \cite[Section 7.10]{Stanley_EC2_1999}.

\begin{defn}
The $r$-Kostka cone is the $(2r-1)$-dimensional polyhedral cone formed by taking the convex hull in $\R^{2r}$ of the points $(\lambda_1,\dots,\lambda_r,\mu_1,\dots,\mu_r) \in \Z_{\geq 0}^{2r}$ where $\lambda = (\lambda_1,\dots,\lambda_r)$ and $\mu = (\mu_1,\dots,\mu_r)$ are both elements of $\Par_r(n)$ for some $n$ and where  $\lambda$ dominates $\mu$.  
\end{defn}

Note that the Kostka cone is pointed, i.e., contains no nontrivial linear subspace. The Kostka cones can be viewed as nested via the following observation.

\begin{obs}\label{obs: kostka slice}
The cone $\Kostka_r$ is combinatorially equivalent to the codimension-$2$ face of $\Kostka_{r+1}$ obtained by intersecting with the hyperplane given by the equation $\mu_r = 0$.
\end{obs}

\subsubsection{Facets}
The bounding hyperplanes of $\Kostka_r$ are simple to describe by examining the required inequalities satisfied by individual entries of each element.

\begin{obs}\label{obs: bounding hyperplanes} The Kostka cone is bounded by the following hyperplanes for $1 \leq i < r$:
\begin{align*}
    H_i &= \{(\lambda,\mu) \in \R^{2r}: \lambda_i = \lambda_{i+1}\}\,, \\
    H_r &= \{(\lambda,\mu) \in \R^{2r}:\lambda_r = 0\}\,,\\
    \widehat H_i &= \{(\lambda,\mu) \in \R^{2r}: \mu_i = \mu_{i+1}\}\,, \text{ and}\\
    J_i &= \left\{(\lambda,\mu) \in \R^{2r}:\sum_{j=1}^i \lambda_j = \sum_{k=1}^i \mu_k\right\}\,.
\end{align*}
\end{obs}

\begin{rem}\label{rem: num facets}
It is straightforward to check that each of these hyperplanes intersects $\Kostka_r$ along a facet, and these facets are distinct when $r > 2$.  Thus, $\Kostka_r$ has $3r-2$ facets for $r > 2$.
\end{rem}

\subsubsection{The Kostka polytope}
Since a large portion of this work concerns the face structure of $\Kostka_r$, it is often more convenient to work with a polytopal slice of this cone.  

\begin{defn}
    Let $\Kpoly_r$ be the $(2r-2)$-dimensional polytope obtained by intersecting $\Kostka_r$ with the affine hyperplane $\left\{\sum_{i=1}^r (\lambda_i +\mu_i) = 1\right\}$. 
\end{defn}

In other words, $\Kpoly_r$ is the set of points $(\lambda,\mu)$ in $\Kostka_r$ such that $\lambda$ and $\mu$ each have entries summing to $\frac{1}{2}$.  Since we are interested only in the combinatorial type of $\Kpoly_r$, we could have equivalently intersected $\Kostka_r$ with any affine hyperplane nontrivially intersecting all faces of $\Kostka_r$ except the origin.  

\begin{obs}\label{obs: face correspondence}
The $d$-faces of $\Kpoly_r$ are in bijection with the $(d+1)$-faces of $\Kostka_r$.  In particular, each $(d+1)$-face of $\Kostka_r$ is obtained by taking all points along any ray emanating from the origin and passing through some fixed $d$-face of $\Kpoly_r$. Thus, the vertices of $\Kpoly_r$ correspond to the extremal rays of $\Kostka_r$.
\end{obs}

\subsubsection{Extremal Rays}
The extremal rays of $\Kostka_r$ were described in \cite{gao2021kostka}.  In particular, we have

\begin{prop}{\cite[Proposition 4.1, Corollary 1.7]{gao2021kostka}}\label{prop: extremal rays}
Let $a,b,\ell$ satisfy $0 \leq \ell < b \leq a\leq r$.  Then
\begin{align*}(\lambda,\mu) &= \left( \underbrace{a-\ell,\dots,a-\ell}_{b}, 0 \dots, 0; \underbrace{a - \ell,\dots,a - \ell}_{\ell}, \underbrace{b - \ell,\dots,b - \ell}_{a - \ell}, 0,\dots,0\right)\\
&=((a-\ell)^b, 0^{r-b});\ 
((a-\ell)^\ell, (b-\ell)^{a-\ell}, 0^{r-a}))\,,
\end{align*}
generates an extremal ray of $\Kostka_r$, and all extremal rays are generated by such an element. In particular, the number of extremal rays of $\Kostka_r$ is 
$\binom{r}{3}+\binom{r}{2}+\binom{r}{1}$.
\end{prop}

\begin{exmp}
Let $r = a = 5$, $b = 4$, and $\ell = 2$. Then

\[(\lambda,\mu)= \left((3,3,3,3,0), (3,3,2,2,2)\right) = \left(\,\tableau{ \ &\ &\ \\\ &\ &\ \\ \ &\ &\ \\ \ &\ &\ }\ ,\ \tableau{\ &\ &\ \\\ &\ &\ \\ \ 
 &\ \\ \ &\ \\ \ &\ }\,\right)\]
generates an extremal ray of $\Kostka_5$.
\end{exmp}

\begin{defn}\label{defn: extremal ray label}
We say that the extremal ray in \autoref{prop: extremal rays} is \emph{labeled} by the triple $(a,b,\ell)$ whenever $a \neq b$.  Whenever $a = b$, the extremal ray in \autoref{prop: extremal rays} is not dependent on the choice of $\ell$, and we say it is \emph{labeled} by the triple $(a,a,a)$.  
We also say that the corresponding vertex of $\Kpoly_r$ (using \autoref{obs: face correspondence}) is \emph{labeled} by the same triple.
\end{defn}

\begin{exmp}
The seven extremal rays of $\Kostka_3$ are labeled by the triples $(1,1,1)$, $(2,1,0)$, $(2,2,0)$, $(3,1,0)$, $(3,2,0)$, $(3,2,1)$, and $(3,3,3)$.
\end{exmp}

\begin{rem}
Note that our usage of the parameters $a,b,\ell$ differs from the convention in \cite{gao2021kostka}; in particular, we relabel their parameter $a + \ell$ by $a$ and $b + \ell$ by $b$.  While the choice of label $(a,a,a)$ may seem arbitrary for the case when $a = b$, this choice simplifies the statement of \autoref{lem: parameter bound}.
\end{rem}

\subsection{Hilbert Bases}

Let $C \subseteq \R^d$ be a rational convex polyhedral cone.  By Gordan's Lemma \cite[Theorem 16.4]{schrijver1998theory}, 
there exists a finite set $\calH(C) \subseteq C \cap \Z^d$, such that
\begin{itemize}
\item every integral point of $C$ can be expressed as a nonnegative integer combination of points in $\calH(C)$, and
\item $\calH(C)$ has minimal cardinality with respect to the first property.  
\end{itemize}

In the case that $C$ is pointed, the set $\calH(C)$ is unique and is known as the \emph{Hilbert basis} of $C$.  Moreover, an element of $C \cap \Z^d$ is in the Hilbert basis if and only if it is \emph{irreducible}, i.e., cannot be expressed as a nonnegative integer combination of any other integral points of $C$; otherwise it is called \emph{reducible}.  See \cite[Section 16.4]{schrijver1998theory} for further background.

\begin{rem}\label{rem: reducible condition}
Since $\Kostka_r$ is pointed and has integral points corresponding to pairs in $\Par_r(n)$, we can express Hilbert basis membership in terms of the partitions.  Namely, an element $(\lambda,\mu) \in \Kostka_r \cap \Z^{2r}$ is a Hilbert basis element if and only if no nontrivial subset of the columns of $\lambda$ has total size equal to a subset of the columns of $\mu$.
\end{rem}

\section{The Maximum Number of Vertices of a Face}\label{sec: max face}
In this section, we look at the maximum number of vertices contained in a $d$-face of the polytope $\Kpoly_r$.  Equivalently (see \autoref{obs: face correspondence}), we look at the maximum number of extremal rays contained in a $(d+1)$-face of the cone $\Kostka_r$. We give a uniform upper bound on this quantity for fixed $d$, and furthermore show that this upper bound is exact for $r > d + 1$.

\begin{defn}
    For integers $r \geq 1$ and $0 \leq d \leq 2r-2$, let $m(r,d)$ denote the maximum number of vertices in a $d$-dimensional face of the polytope $\Kpoly_r$.  Let $m(d)$ denote the maximum number of vertices of a $d$-face in any polytope $\Kpoly_j$ over all choices of $j \geq 1$.
\end{defn}

By \autoref{obs: kostka slice}, we have that  $m(r,d)$ is non-decreasing as a function in $r$.  Moreover, since any proper face can be extended to a face of higher dimension, the function $m(r,d)$ is strictly increasing in $d$.  \autoref{table: m(r,d)} depicts some values of $m(r,d)$.

\begin{table}[ht]
\begin{tabular}{l|lllllllllllllll}
  & 2 & 3 & 4 & 5  & 6  & 7  & 8  & 9  & 10 & 11 & 12 & 13 & 14  & 15  & 16   \\ \hline
2 & 3 &   &    &    &    &    &    &    &    &    &    &     &     &     &     \\
3 & 4 & 6 & 7  &    &    &    &    &    &    &    &    &     &     &     &     \\
4 & 4 & 7 & 10 & 13 & 14 &    &    &    &    &    &    &     &     &     &     \\
5 & 4 & 8 & 11 & 15 & 19 & 24 & 25 &    &    &    &    &     &     &     &     \\
6 & 4 & 8 & 12 & 17 & 23 & 28 & 34 & 40 & 41 &    &    &     &     &     &     \\
7 & 4 & 8 & 12 & 18 & 25 & 32 & 40 & 48 & 55 & 62 & 63 &     &     &     &     \\
8 & 4 & 8 & 12 & 18 & 27 & 34 & 45 & 53 & 64 & 75 & 83 & 91  & 92  &     &     \\
9 & 4 & 8 & 12 & 18 & 27 & 36 & 46 & 58 & 69 & 82 & 95 & 110 & 119 & 128 & 129
\end{tabular}
\caption{Some values of $m(r,d)$, the maximum number of vertices in a $d$-face of $\Kpoly_r$, are shown, appearing in the row labeled by $r$ and the column labeled by $d$.}
\label{table: m(r,d)}
\end{table}

\begin{rem}
    Note that $m(d)$ is a priori not guaranteed to exist, but \autoref{cor: max face size} shows that it is well-defined.
\end{rem}

Our main result is an exact calculation of $m(d)$, which in turn gives an upper bound on $m(r,d)$.  Using the language of $m(d)$ and $m(r,d)$, we restate the result stated in \autoref{thm: max face size}.

\begin{cor}\label{cor: max face size}
    For $r > d + 1$, we have 
    $$m(r,d) = m(d) = \prod_{i=1}^3\left\lfloor \frac{d+2 + i}{3} \right\rfloor \,.$$
\end{cor}

\begin{rem}
The values of $m(d)$ appear as the sequence \href{https://oeis.org/A006501}{A006501} in the OEIS \cite{oeis}, with generating function $\displaystyle \frac{1 + x^2}{(1-x)^2(1-x^3)^2}$.  The quantity $m(d)$ can be alternatively characterized as the maximum product of three positive integers summing to $d+3$. 
\end{rem}
From \autoref{obs: bounding hyperplanes} and \autoref{prop: extremal rays} the following is clear.
\begin{prop}\label{prop: hyp rules}
    Let $v$ be a vertex of $\Kpoly_r$ labeled by the triple $(a,b,\ell)$.  Then
    \begin{itemize}
       \item $v \in H_i$ if and only if $b \neq i$,
        \item $v \in \widehat H_k$ if and only if $a \neq k$ and $\ell \neq k$. 
        \item $v \in J_j$ if and only if $j \leq \ell$, $j \geq a$, or $a = b$.
    \end{itemize}
\end{prop}

\begin{thm}\label{thm: max face ub}
    A $d$-dimensional face of $\Kpoly_r$ has at most $ \prod_{i=1}^3\left\lfloor \frac{d+2 + i}{3} \right\rfloor $ vertices.
\end{thm}
\begin{proof}
    Let 
    $$F = \Kpoly_r \cap \left(\bigcap_{i \in I} H_i \right) \cap \left( \bigcap_{j \in J} J_j \right) \cap \left(\bigcap_{k \in K} \widehat H_k \right)$$
    be a $d$-dimensional face of $\Kpoly_r$, where $I \subseteq \{1,2,\dots,r\}$ and $J,K \subseteq \{1,2,\dots,r-1\}$ are (possibly empty) index sets.  We can furthermore assume that the set of hyperplanes is chosen minimally to have this intersection, i.e., $|I| + |J| + |K| = 2r - 1 - d$.  

    We are interested in bounding the possible triples $(a,b,\ell) \in \Z_{\geq 0}$ labeling the vertices of $F$.   According to \autoref{prop: hyp rules}, such a triple must satisfy that $b \notin I$, $a,\ell \notin K$, and an element of $J$ is weakly between $a$ and $\ell$ only if $a = b = \ell$.  Let $F_1$ be the set of triples $(a,b,\ell)$ meeting these conditions.  The minimality condition implies that, for any elements $j < j' < j''$ of $J \cup \{0,r\}$, there must be some $(a,b,\ell) \in F_1$ such that $j \leq \ell < j' < a \leq j''$.  That is, the sets $\{j,j+1,\dots, j'-1\} \cut K$ and $\{j-1,j,\dots, j'\} \cut K$ are nonempty for any elements $j < j'$ in $J \cup \{0,r\}$.  

    Fix $b \notin I$.  Let $z_1(b) = |\{a : (a,b,\ell) \in F_1 \text{ for some } a,\ell\}|$ and 
    $z_2(b) = |\{\ell : (a,b,\ell) \in F_1 \text{ for some } a,\ell\}|$.  If $b \in J$, then we must have $a = b = \ell$, so $z_1(b) + z_2(b) \leq 2$.  If $b \notin J$, since each $j \in J$ has an element of $\{0,\dots,r\} \cut K$ on either side of it, we have $z_1(b) + z_2(b) \leq r + 1 - |J| - |K|$.  Thus, summing over our choices for $b$, we have
    \begin{align*}
        |F_1| &\leq \sum_{b \in [r] \cut I} z_1(b) \cdot z_2(b)\\
        &\leq \sum_{b \in [r] \cut I} \left\lfloor \frac{r+1-|J|-|K|}{2} \right\rfloor \cdot \left\lfloor \frac{r+2-|J|-|K|}{2} \right\rfloor\\ &\leq (r-|I|)\cdot \left\lfloor \frac{r+1-|J|-|K|}{2} \right\rfloor \cdot \left\lfloor \frac{r+2-|J|-|K|}{2} \right\rfloor\,,
    \end{align*}
    where, in the second step, we replace the summand by the maximum value of the product of two numbers summing to $r + 1 - |J|-|K|$.  The sum of the three factors in the final expression is 
    $$2r + 1 - |I| - |J| - |K| = d + 3\,,$$
    so their product is at most $\displaystyle \prod_{i=1}^3\left\lfloor \frac{d+2 + i}{3} \right\rfloor \,.$
    This yields the desired upper bound.
\end{proof}

Via a construction, we can prove a lower bound on $m(r,d)$.  
\begin{thm}\label{thm: max face lb}
     Suppose $r > d+1$.  Given any positive integers $z_1,z_2,z_3$ summing to $d + 3$, the intersection 
    $$F = \Kpoly_r \cap \left(\left(\bigcap_{i = 1}^{z_1 - 1} H_i\right) \cap \left(\bigcap_{j = z_1 + z_2}^{r} H_j \right) \cap  \left(\bigcap_{k = z_1}^{r - z_3} \widehat H_k\right)\right)$$
    is a face of $\Kpoly_r$ of dimension at most $d$ with $z_1z_2z_3$ vertices.
\end{thm}
\begin{proof}
      We begin by determining the set of vertices contained in $F$.  Let $v$ be a vertex of $\Kpoly_r$ labeled by $(a,b,\ell)$.  We have $v \in \left(\bigcap_{i=0}^{z_1-2} H_i \right) \cap  \left(\bigcap_{j = z_1 + z_2}^{r} H_j \right)$ if and only if $z_1 \leq b \leq z_1 + z_2 - 1$.  Similarly, we have $v \in   \bigcap_{k = z_1}^{r - z_3} \widehat H_k$ if and only if $a,\ell \not\in \{z_1,\dots,r-z_3\}$.  By assumption, we have $r - z_3 \geq z_1 + z_2 - 1$ and $\ell \leq b$, hence $v \in F$ if and only if 
      $$0 \leq \ell < z_1 \leq b \leq z_1 + z_2 - 1 \leq r - z_3 < a \leq r\,.$$      
    The ranges for $\ell$, $b$, and $a$ are disjoint and of sizes $z_1$, $z_2$, and $z_3$, respectively.  Therefore, there are $z_1 \cdot z_2 \cdot z_3$ vertices of $\Kpoly_r$ contained in $F$, each associated to a triple $(a,b,\ell)$ satisfying the inequalities above.

    It remains to show that dimension of $F$ is at most $d$.  This follows because any element $(\lambda,\mu)$ in $F$ lies in the affine subspace of $\R^{2r}$ where
    $$\lambda_1 = \lambda_2 = \cdots = \lambda_{z_1},\;\lambda_{z_1+z_2} = \cdots = \lambda_{r},\; \mu_{z_1} = \cdots = \mu_{r-z_3},\text{ and } \sum_{i=1}^r \lambda_i = \sum_{j=1}^r \mu_j = \frac{1}{2}\,,$$
    which has dimension $2r - (z_1-1) - (r-z_1-z_2 + 1) - (r-z_3 - z_1 + 1) - 2 = d$.
\end{proof}

\begin{proof}[Proof of \autoref{cor: max face size}]
    The upper bound follows directly from \autoref{thm: max face ub}.  For the lower bound, consider the face constructed in \autoref{thm: max face lb} with $z_i = \left\lfloor \frac{d + 1 + i}{3} \right\rfloor$.  Since this face achieves the upper bound on the number of vertices in a $d$-face from \autoref{thm: max face ub} and $m(r,d)$ is strictly increasing in $d$, we can conclude that this face has dimension exactly $d$. 
\end{proof}

\section{Characterization of Edges}\label{sec: edge description}
Here we present a procedure for characterizing the faces of a fixed dimension in the Kostka polytope $\Kpoly_r$, where $r$ can vary.  We carry out this characterization explicitly for dimension $1$.  This characterization yields an enumeration of the faces of these dimensions, which is handled in the following section.  It seems very feasible that these methods could be extended to higher dimensions, though the conditions seem to get increasingly complicated.

\begin{prop}\label{prop: minimal face conditions}
    The minimal face of $\Kpoly_r$ containing a set of vertices with labels $\{(a_i,b_i,\ell_i)\}_{1 \leq i \leq m}$ is formed by the set of all vertices whose label $(a,b,\ell)$ satisfies that
    \begin{enumerate}
        \item\label{cond: b} $b$ is an element of $\bigcup_{i=1}^m \{b_i\}$,
        \item\label{cond: a ell} $\ell$ and $a$ are both elements of $\bigcup_{i=1}^m \{\ell_i,a_i\}$,
        \item\label{cond: open interval} the open interval $(\ell,a)$ is contained in  $\bigcup_{i=1}^m (\ell_i,a_i)$, and 
        \item\label{cond: ordering} $0 \leq \ell < b < a \leq r$ or $a = b = \ell$.
    \end{enumerate}
\end{prop}
\begin{proof}
Comparing these conditions to those in \autoref{prop: hyp rules}, we see that these conditions precisely encode that the vertex labeled by $(a,b,\ell)$ is contained in all hyperplanes that contain the vertices with labels $\{(a_i,b_i,\ell_i)\}_{1 \leq i \leq m}$.
\end{proof}

\begin{rem}
For convenience, when considering the labels of a list of vertices, we  follow the convention that the labels are ordered lexicographically.
\end{rem}

We will now show that whether a collection of vertices is the vertex set of some face of the Kostka cone depends only on the cell of the braid arrangement that the vertex label list lies in, i.e., the relative order of the vertex label entries.  We say that two tuples $(x_1,\dots,x_n),(y_1,\dots,y_n) \in \Z^n$ are \emph{order-isomorphic} provided that $x_i > x_j$ if and only if $y_i < y_j$ for any $i,j \in \{1,\dots,n\}$.

\begin{lem}\label{lem: order isom face}
    Suppose we have a pair of order-isomorphic tuples ${(a_1,b_1,\ell_1,\dots,a_m,b_m,\ell_m)}$ and ${(a'_1,b'_1,\ell'_1,\dots,a'_m,b'_m,\ell'_m)}$ in $\{0,\dots,r\}^3$ such that the triples $(a_i,b_i,\ell_i)$ and $(a'_i,b'_i,\ell'_i)$ are labels of vertices of $\Kpoly_r$.  Then the vertices labeled by $\{(a_i,b_i,\ell_i)\}_{1 \leq i \leq m}$ form the vertex set of a $d$-dimensional face of $\Kpoly_r$ if and only if the vertices labeled by $\{(a'_i,b'_i,\ell'_i)\}_{1 \leq i \leq m}$ do.
\end{lem}
\begin{proof}
  In order to determine if a set $V$ of vertices in $\Kpoly_r$ labeled by $\{(a_i,b_i,\ell_i)\}_{1 \leq i \leq m}$ is the vertex set of a $d$-face of $\Kpoly_r$, we test whether any other vertex of $\Kpoly_r$ lies in the intersection of the hyperplanes containing $V$.  In order to lie in this intersection, the new vertex labeled by $(a,b,\ell)$ must satisfy the conditions of \autoref{prop: minimal face conditions}.
 
    These conditions, and hence the existence of such a tuple, only depend on the order-isomorphism class of the tuple $(a_1,b_1,\ell_1,\dots,a_m,b_m,\ell_m)$.  Moreover, all vertex sets corresponding to a given ordering have convex hulls of the same dimension, since the set of bounding hyperplanes of $\Kpoly_r$ containing a vertex is determined entirely by this ordering. 
\end{proof}

Thus, in order to determine if a set of vertices is the vertex set of some face of $\Kpoly_r$, it is sufficient to test this for any set of vertices with an order-isomorphic list of labels.  That is, a list in $\{0,\dots,r\}^{3m}$ being the set of labels of a face of $\Kpoly_r$ is constant across open cells of the braid arrangement $\mathcal{B}_{3m}$.  We can combine this fact with the well-known Upper Bound Theorem for polytopes, proved by McMullen \cite{mcmullen_1970} in 1970  (see \cite[Chapter 2, Section 3]{stanley1996CCA} for more details).  This yields an upper bound on the dimension of open cells in $\mathcal{B}_{3m} \cap \{0,\dots,r\}^{3m}$ that correspond to vertex labels of $d$-faces of $\Kpoly_r$.  We state the Upper Bound Theorem under the additional assumption that the face dimension is less than half the polytope dimension, which is sufficient for our purposes.

\begin{thm}\label{thm: upper bound theorem}[Upper Bound Theorem, {\cite{mcmullen_1970}}]
For $0 \leq i \leq \left\lfloor \frac{m}{2}\right\rfloor$, the number of $i$-faces of an $m$-polytope with $n$ vertices is at most 
$\binom{n}{i+1}$.  Moreover, this bound is realized by $\cyclic(n,m)$, the $m$-dimensional cyclic polytope with $n$ vertices.
\end{thm}

We now prove an upper bound on the number of distinct values of the triples labeling the vertices of a face of fixed dimension in $\Kpoly_r$.  Of course, we already have an upper bound of $3\prod_{i=1}^3\left\lfloor \frac{d+2 + i}{3} \right\rfloor$ from \autoref{thm: max face ub}, which bounds the number of vertices.  However, we can obtain a tight bound using the upper bound theorem. 

\begin{lem}\label{lem: parameter bound}
If the vertices of a $d$-face of $\Kpoly_r$ are labeled by $\{(a_i,b_i,\ell_i)\}_{1 \leq i \leq n}$, then there are at most $3d+3$ distinct values among the parameters $a_1,b_1,\ell_1,\dots,a_n,b_n,\ell_n$.
\end{lem}
\begin{proof}
Let $t = |\{a_1,b_1,\ell_1,\dots,a_n,b_n,\ell_n\}|$.  By \autoref{lem: order isom face}, the number of $d$-faces of $\Kpoly_r$ generated by tuples order-isomorphic to $(a_1,b_1,\ell_1,\dots,a_n,b_n,\ell_n)$ is $\binom{r+1}{t} = \Theta(r^{t})$.  

We now apply the upper bound theorem for polytopes (see \autoref{thm: upper bound theorem}).  Since the number of vertices is $\binom{r}{3} + \binom{r}{2} + \binom{r}{1}$ by \autoref{prop: extremal rays}, then the number of faces of dimension $d$ is bounded by the corresponding number of $d$-faces of the cyclic $2r$-polytope with $\binom{r}{3} + \binom{r}{2} + \binom{r}{1}$ vertices.  This quantity is asymptotically  $\Theta(r^{3d + 3})$. 
Therefore, we must have $t \leq 3d + 3$, as desired.
\end{proof}

It follows that in order to characterize the $d$-dimensional faces of $\Kpoly_r$ for arbitrary $r$, one must merely determine the $d$-faces of $\Kpoly_{3d+3}$.  The faces of $\Kpoly_r$ are then those whose label sets are order-isomorphic to a label set of a $d$-face of $\Kpoly_{3d+3}$.

\begin{thm}\label{thm: 2d faces}
Let $u$ and $v$ be vertices of $\Kpoly_r$ labeled $(a,b,\ell)$ and $(a',b',\ell')$, where $a-b \leq a' - b'$.  Then $\{u,v\}$ is a face of $\Kpoly_r$ if and only if
\begin{enumerate}[(1)]
    \item $a = b$ and at least one of the following holds:
    \begin{enumerate}[(i)]
        \item $a' = b'$,
        \item $a = b'$,
        \item $a \geq a'$, or
        \item $\ell' \geq a$.
    \end{enumerate}
    \item $a \neq b$ and at least one of the following holds:
    \begin{enumerate}[(i)]
        \item two of the three equalities $a = a'$, $b = b'$, and $\ell = \ell'$ hold,
        \item $\ell \geq a'$, or
        \item $\ell' \geq a$.
    \end{enumerate}
\end{enumerate}
\end{thm}
\begin{proof}
By \autoref{lem: order isom face} and \autoref{lem: parameter bound}, it is enough to check that this is the case for the vertices of $\Kpoly_6$.  This can be readily completed with the aid of a computer.
\end{proof}

By similarly examining the $2$-faces of $\Kpoly_9$, one could determine a characterization of the $2$-faces of all Kostka polytopes.  While the conditions seem rather complex, we shall see in the next section that these methods yield nice enumerative results.

\section{Enumeration of Faces of a Fixed Dimension}\label{sec: enumeration}

In this section, we derive formulas for the number of faces of a fixed dimension $d$ in $\Kpoly_{r}$ for $d = 1,2,3$.  We then asymptotically determine the number of $d$-faces of $\Kpoly_r$ for arbitrary $d$.  As mentioned in \autoref{thm: upper bound theorem}, it is well known that the number of $d$-faces of a $k$-polytope with $n$ vertices is maximized by the cyclic polytope $\cyclic(n,k)$  for sufficiently large $k$.  We show that, as $r$ increases, the number of $d$-faces of $\Kpoly_{r}$ grows asymptotically at the same rate as the number of $d$-faces of $\cyclic\left(\binom{r}{3} + \binom{r}{2} + \binom{r}{1},2r-2\right)$ up to a constant factor depending on $d$, and we furthermore determine this constant for all $d$ (see \autoref{cor: asymptotic cyclic}).

\begin{defn}
    Let $f_d(r)$ denote the number of $d$-dimensional faces of $\Kpoly_r$. 
\end{defn}

In the previous section, we showed that whether a set of $m$ vertices of $\Kpoly_r$ forms the vertex set of a face depends only on the order-isomorphism class of the vertex labels (see \autoref{lem: order isom face}).  In other words, it depends only on the cell of the braid arrangement $\mathcal{B}_{3m}$ that the list of $m$ vertex label triples lies in. By examining the integer points in each cell, we obtain the following lemma.   

\begin{lem}\label{cor: face polynomial degree}
    The function $f_d(r)$ is a polynomial in $r$ of degree at most $3d+3$ and has a positive integer expansion in terms of the of basis $\left\{\binom{r}{k}\right\}_{0\leq k \leq 3d+3}$.  
\end{lem}
\begin{proof}
The number of integer points in any collection of cells of $\mathcal{B}_{3m} \cap \{0,\dots,r\}^{3m}$ has a positive integer expansion in the basis $\left\{\binom{r}{k}\right\}_{1\leq k \leq 3m}$.  Thus, it follows directly from \autoref{lem: order isom face} that $f_d(r)$ is a polynomial with a positive integer expansion in terms of the of basis $\left\{\binom{r}{k}\right\}_{0\leq k}$.  The claim about the degree then follows from \autoref{lem: parameter bound}.
\end{proof}

\begin{thm}\label{thm: face recursion formula}
Fix $d \geq 0$.  Setting $d_{\min} = \left\lfloor \frac{d+3}{2} \right\rfloor$, we have 
    $$f_d(r) = \sum_{k=d_{\min}}^{3d+3} \alpha_k \binom{r}{k}$$
    where $\alpha_{k}= f_d(k) - \left(\sum_{j=d_{\min}}^{k-1} \binom{k}{j} a_j\right)$ for $k > d_{\min}$ and 
    $$\alpha_{d_{\min}} = f_d(d_{\min}) = \begin{cases}
    3d-2 & \text{ if $d$ odd and } d > 1\,,\\
    1 & \text{ if $d$ even,}\\
    3 & \text{if $d = 1$}\,.
    \end{cases}$$
\end{thm}
\begin{proof}
If $d$ is odd, then the value of $\alpha_{d_{\min}}$ is the number of facets of $\Kpoly_{d_{\min}}$, which we calculate in \autoref{rem: num facets}.  If $d$ is even, then $\alpha_{d_{\min}}$ is the number of top-dimensional faces, which is $1$ since $\Kpoly_r$ is a polytope.  The recursive formula for the other values of $\alpha_k$ follows from \autoref{lem: order isom face}, \autoref{obs: kostka slice}, and \autoref{cor: face polynomial degree} by evaluating $f_d(k)$ as a sum of terms of the form $\alpha_k\binom{r}{k}$.  
\end{proof}

Thus, if one can compute the values $f_d(0),\dots,f_d(3d+3)$, then \autoref{thm: face recursion formula} implies that we can determine the entire function $f_d(r)$.  Using SageMath, we were able to compute some initial terms of $f_d(r)$ (see \autoref{table: face numbers}).  The number of vertices, $f_0(r)$, is also shown in \autoref{table: face numbers} for $r \leq 13$, with the general formula given in \cite{gao2021kostka}.

\begin{table}[ht]
\centering
\begin{tabular}{l|lllllllllllll}
  & 1 & 2 & 3  & 4  & 5   & 6    & 7    & 8     & 9     & 10     & 11     & 12     & 13     \\ \hline
0 & 1 & 3 & 7 & 14 & 25 & 41 & 63  &   92    &   129    &      175  &     231   &  298      &   377     \\
1 & 0 & 3 & 16 & 52 & 132 & 288  & 567  &  1036     &   1788    &  2949      &     4686   &   7216     &     10816   \\
2 & 0 & 1 & 16 & 89 & 328 & 961  & 2427 & 5517  & 11584 & 22846  &    42812    &     76868   &  133068  \\
3 & 0 & 0 & 7  & 81 & 466 & 1898 & 6253 & 17803 & 45502 & 106946 & 234964 & 488229 & 967863
\end{tabular}
\caption{The values of $f_d(r)$, the number of $d$-faces in $\Kpoly_r$, are shown for the cases where $0 \leq d \leq 3$ and $1 \leq r \leq 13$.  Here, $d$ is given by the row label and $r$ is given by the column label.}
\label{table: face numbers}
\end{table}

These computations allow us to derive formulas for the number of $d$-faces of $\Kpoly_r$ for $d = 1,2,3$, given in \autoref{thm: face numbers}.

\begin{proof}[Proof of \autoref{thm: face numbers}]
    Each formula can be obtained by applying \autoref{thm: face recursion formula} to the values in a fixed row of \autoref{table: face numbers}.
\end{proof}

We now shift our focus to determining the asymptotic behavior of the function $f_d(r)$.  To achieve this, we determine the degree and leading coefficient of the polynomial $f_d(r)$.

\begin{lem}\label{lem: simplex construction}
Fix positive integers $d,r$ such that $r \geq 3d+3$.  If a set of $d+1$ vertices in $\Kpoly_r$ with labels $\{(a_i,b_i,\ell_i)\}_{1 \leq i \leq d+1}$ satisfies that the intervals $[\ell_i,a_i]$ are all disjoint, then it is the vertex set of a $d$-face of $\Kostka_r$.
\end{lem}
\begin{proof}
We prove this by induction on $d$, with the base case $d = 1$ following from \autoref{thm: 2d faces}.  We first show that such a set of vertices is the vertex set of a face of $\Kpoly_r$, and then determine its dimension.  We prove the former by showing there is no other vertex in the minimal face $F$ containing the vertices labeled by $\{(a_i,b_i,\ell_i)\}_{1 \leq i \leq d+1}$ via the conditions of \autoref{prop: minimal face conditions}.  Let $(a,b,\ell)$ be the label of a vertex in $F$.  By Condition (\ref{cond: a ell}), the parameters $a,\ell$ must be chosen from within intervals $[\ell_i,a_i]$.  If $a$ and $\ell$ are chosen from different intervals, then in order to satisfy Condition (\ref{cond: open interval}), we must have $a = \ell + 1$.  However, then $b$ cannot be chosen to satisfy Condition (\ref{cond: ordering}).  On the other hand, if $a$ and $\ell$ are chosen within the same interval $[\ell_j,a_j]$, then Condition (\ref{cond: a ell}) implies that $a = a_j$ and $\ell = \ell_j$.  But then Condition (\ref{cond: b}) and the required ordering of $\ell$, $b$, and $a$ imply that we also have $b = b_j$, so $(a,b,\ell)$ was in the original list of vertex labels.  Hence, the minimal face containing the vertices labeled by  $\{(a_i,b_i,\ell_i)\}_{1 \leq i \leq d+1}$ contains no other vertices, so these form the vertex set of $F$.

The fact that the dimension of $F$ is $d$ follows from the induction.  In particular, we know that the vertices with labels $\{(a_i,b_i,\ell_i)\}_{1 \leq i \leq d}$ form the vertex set of a face of dimension $d - 1$.  Since we have added one additional vertex and formed another face of $\Kpoly_r$, the dimension of $F$ must be $d$.
\end{proof}

\begin{lem}\label{lem: only simplex}
Fix positive integers $d,r$ such that $r \geq 3d+3$, and suppose $L$ is the set of vertex labels of a $d$-face of $\Kpoly_r$.  Then either 
\begin{enumerate}[(i)]
    \item $F$ is a simplex whose vertex labels satisfy the conditions of \autoref{lem: simplex construction}, or
    \item there are at most $3d+2$ distinct values among the vertex label entries.
\end{enumerate} 
\end{lem}
\begin{proof}
We proceed by induction on $d$, with the base case $d = 1$ following from \autoref{thm: 2d faces}.  

Let $F$ be a $d$-face of $\Kpoly_r$, and let $L = \{(a_i,b_i,\ell_i)\}_{1 \leq i \leq n}$ be the set of labels of the vertices of $F$.  Let $t$ denote the number of distinct values among the label entries $a_i$, $b_i$, and $\ell_i$.  Fix a bounding hyperplane of $F$ of type $H_i$ or $\widehat H_i$ (see \autoref{obs: face correspondence} for hyperplane descriptions), i.e., 
$$H \in \big\{H_i : 1 \leq i \leq r,\; \dim(H_i \cap F) = d-1\big\}  \cup \big\{\widehat H_i : 1 \leq i \leq r,\; \dim(\widehat H_i \cap F) = d-1\big\}\,,$$
such that the number of vertices of $F$ contained in $H$ is minimal.  

Suppose $H = H_j$ (the case for $\widehat H_j$ proceeds analogously).  By \autoref{prop: hyp rules}, the vertices of $F$ that are contained in $H$ are precisely those whose label $(a_i,b_i,\ell_i)$ does not have $b_i = j$.  We now consider the number of distinct values among the label entries of the vertices in $F \cap H$.  If a label entry $m$ appears among the vertices of $F$ but not $F \cap H$, then all vertices whose label contains the entry $m$ must also contain the entry $j$.  Moreover, by the minimality condition, $F \cap \widehat H_m$ has at least as many vertices as $F \cap H$, so $F \cap \widehat H_m = F \cap H_j$.  Since each label has three entries, it is either the case that 
\begin{enumerate}[(a)]
\item there are at most two label entries that appear among the vertices of $F$ but not $F \cap H$, or
\item there are exactly three label entries that appear among the vertices $F$ but not $F \cap H$, and these entries appear in the label of a unique vertex of $F$.
\end{enumerate}
In Case (a), the face $F \cap H$ is then a $(d-1)$-dimensional face of $\Kpoly_r$ with at least $t-2$ distinct entries among the labels of its vertices.  By the inductive hypothesis, this implies $t \leq 3d + 2$.

In Case (b), the face $F \cap H$ is a $(d-1)$-dimensional face of $\Kpoly_r$ with $t-3$ distinct entries among the labels of its vertices and one fewer vertex than $F$.  Thus, by the inductive hypothesis, we have $t \leq 3d+3$.  

It remains to show that, if $t = 3d + 3$ in Case (b), then $F$ is a simplex satisfying the conditions of \autoref{lem: simplex construction} (up to reordering of the vertices).  In this case, the inductive hypothesis implies that $F \cap H$ is a simplex whose labels satisfy the conditions of  \autoref{lem: simplex construction}.  So it is enough to show that the label $(a,b,\ell)$ of the unique vertex of $F$ that is  not in $F \cap H$ satisfies $a < \ell'$ or $a' < \ell$ for any label $(a',b',\ell')$ of a vertex of $F \cap H$.  This must hold because otherwise $(\ell,b',a')$ or $(\ell',b,a)$ is the label of an additional vertex in $F$, contradicting that there is only one vertex of $F$ not contained in $H$.  Therefore, $F$ is indeed a simplex whose labels satisfy the conditions of \autoref{lem: simplex construction}.
\end{proof}

\begin{thm}\label{thm: leading coeff}
The function $f_d(r)$ is a polynomial of degree $3d+3$ with leading coefficient $\frac{1}{(3d+3)!}$.
\end{thm}
\begin{proof}
By \autoref{cor: face polynomial degree}, $f_d(r) = \sum_{k = 1}^{3d+3} \alpha_k \binom{r}{k}$ for nonnegative integers $\alpha_k$.  By \autoref{lem: simplex construction}, the coefficient $\alpha_{3d+3}$ is at least $1$.  By \autoref{lem: only simplex}, the coefficient  $\alpha_{3d+3}$ is at most $1$, and hence we can conclude $\alpha_{3d+3} = 1$.  Expanding this out as a polynomial in $r$, we see that the top degree coefficient is $\alpha_{3d+3}/(3d+3)! = 1/(3d+3)!$.
\end{proof}

\begin{cor}\label{cor: asymptotic cyclic}
For $r \geq 1$, let $n_r = \binom{r}{3} + \binom{r}{2} + \binom{r}{1}$.  We have
$$\lim_{r \to \infty} \frac{f_d(r)}{f_d\left(\cyclic\left(n_r,2r-2\right)\right)} = \frac{6^{d+1}(d+1)!}{(3d+3)!}\,,$$
where $f_d\left(\cyclic\left(n_r,2r-2\right)\right)$ is the number of $d$-faces of the cyclic polytope $\cyclic\left(n_r,2r-2\right)$.
\end{cor}
\begin{proof}
By \autoref{thm: upper bound theorem}, the leading coefficient of the polynomial $f_d\left(\cyclic\left(n_r,2r-2\right)\right)$ is $\frac{1}{6^{d+1}(d+1)!}$.  By \autoref{thm: leading coeff}, the leading coefficient of the polynomial $f_d(r)$ is $\frac{1}{(3d+3)!}$.  Since both polynomials have degree $3d+3$, we can directly compute the limit of their quotient.
\end{proof}

\section{Initial Partition Entries of Hilbert Basis Elements}\label{sec: hilbert}

Lastly, we study some families of Hilbert basis elements of $\Kostka_r^\Z$ in the context of their relation to the face structure.  This work builds upon the ``Width Bound'' proved by Gao, Kiers, Orelowitz, and Yong. See, for example, \cite[Table 1]{gao2021kostka} for the Hilbert basis elements of $\Kostka_4^\Z$.

\begin{thm}[{\cite[Theorem 1.4]{gao2021kostka}}, Width Bound]\label{thm: width bound}
Suppose $(\lambda,\mu)$ is a Hilbert basis element of $\Kostka_r^\Z$.  Then $\lambda_1 \leq r$.  Moreover, if $\lambda_1 = r$ then $\lambda$ and $\mu$ are both rectangles.
\end{thm}

We now further study the initial entries of Hilbert basis elements of $\Kostka_r^\Z$, recalling the following definition.

\begin{defn}
We say that an integer pair $(\lambda_1,\mu_1)$ is \emph{$r$-initial} if there is an element $(\lambda,\mu)$ in the Hilbert basis of $\Kostka_r$ such that $\lambda$ has first element $\lambda_1$ and $\mu$ has first element $\mu_1$.
\end{defn}

By the dominating condition for $\lambda$ and $\mu$, an $r$-initial pair must satisfy $\lambda_1 \geq \mu_1$.  Moreover, note that if $(\lambda_1,\mu_1)$ is $r$-initial, then it is also $r'$-initial for any $r' > r$.  This is because any $(\lambda,\mu) \in \Kostka_r$ can be embedded in $\Kostka_{r'}$ by appending zeroes to $\lambda$ and $\mu$ (see \autoref{obs: kostka slice}), and this map preserves the Hilbert basis elements.

\begin{rem}\label{rem: initial immediate}
    It follows immediately from \autoref{thm: width bound} that
    \begin{itemize}
        \item if $(\lambda_1,\mu_1)$ is $r$-initial then $r \geq \lambda_1$, and 
        \item a pair $(r,\mu_1)$ is $r$-initial if and only if $r$ and $\mu_1$ are coprime.  
    \end{itemize}
\end{rem}

Thus, it remains to determine when $(\lambda_1,\mu_1)$ is $r$-initial for $r > \lambda_1$.  \autoref{prop: extremal rays} implies that the pair $(\lambda_1,\lambda_1)$ is $r$-initial for any $\lambda_1 < r$, as realized by the extremal rays.  It may seem tempting to expect that any pair $(\lambda_1,\mu_1)$ is $(\lambda_1 + 1)$-initial, but there is a counterexample when $\lambda_1 = 14$.  This is currently the only counterexample known to the author.

\begin{exmp}\label{exmp: initial counterexample}
We have checked computationally that $(14,6)$ is not $15$-initial.  Moreover, $r = 15$ is the smallest value such that there is a pair $(r-1,\mu_1)$ with $\mu_1 < r-1$ that is not $r$-initial.  
\end{exmp}

The main result of this section is \autoref{thm: initial main}, which states that a pair $(\lambda_1,\mu_1)$ is $(\lambda_1 + 1)$-initial if $\lambda_1 \geq \mu_1$ and any of the following conditions holds
\begin{itemize}
\item $\lambda_1$ and $\mu_1$ are coprime, or
\item $\lambda_1 + 1$ and $\mu_1$ are coprime, or
\item $\lambda_1 + 1$ and $\mu_1 + 1$ are coprime with $2\mu_1 \geq \lambda_1$.
\end{itemize}

\begin{cor}\label{cor: initial prob}
The probability that a pair of positive integers $\mu_1 < \lambda_1$ satisfies at least one of the conditions of \autoref{thm: initial main} is  
$$\frac{5}{2} \prod_{p\text{ prime}} \left(1 - \frac{1}{p^2}\right) - 2 \prod_{p\text{ prime}} \left(1 - \frac{2}{p^2}\right) + \frac{1}{2} \prod_{p\text{ prime}} \left(1 - \frac{3}{p^2}\right)  > 0.937293\,.$$
\end{cor}
\begin{proof}
The details of this computation are given in the appendix.
\end{proof}

\begin{exmp}
The pairs $(\lambda_1, \mu_1)$ with $\mu_1 < \lambda_1 \leq 30$ for which the conditions of \autoref{thm: initial main} do not hold are $(14, 6)$, $(15, 6)$, $(20, 6)$, $(20, 14)$, $(21, 6)$, $(24, 10)$, $(25, 10)$, $(26, 6)$, $(26, 12)$, $(27, 6)$, $(27, 12)$, and $(27, 21)$.  
\end{exmp}

\begin{thm}\label{thm:: gcd1}
Fix $\lambda_1 > \mu_1$.  Let 
$$r(\lambda_1,\mu_1) = \min\{z \in \mathbb{N} : z \geq \lambda_1,\; \gcd(z,\mu_1) = 1\}\,.$$ 
Then $(\lambda_1, \mu_1)$ is $r(\lambda_1,\mu_1)$-initial. In particular, $(\lambda_1,\mu_1)$ is $(\lambda_1 + \mu_1-1)$-initial.
\end{thm}
\begin{proof}
Let $r = r(\lambda_1,\mu_1)$.  Since some entry among the $\mu_1$ integers $\lambda_1, \dots, \lambda_1 + \mu_1 - 1$ must be equivalent to $1$ modulo $\mu_1$, we have $r \leq \lambda_1 + \mu_1 - 1$.

Let $\lambda$ and $\mu$ be the partitions $$\lambda = (\underbrace{\lambda_1,\dots,\lambda_1}_{\mu_1}) \quad\text{ and }\quad \mu = (\underbrace{\mu_1,\dots,\mu_1}_{r-\mu_1}, \underbrace{\mu_1 - (r - \lambda_1),\dots,\mu_1 - (r-\lambda_1)}_{\mu_1})\,.$$ 

If $(\lambda,\mu)$ were reducible, then, by \autoref{rem: reducible condition}, we could choose a proper subset of the columns of $\lambda$ with the same size as a proper subset of the columns of $\mu$.   The $\mu_1$ columns of $\mu$ are all equivalent to $r$ modulo $\mu_1$, and $\gcd(r,\mu_1) = 1$, and hence  there is no way to choose a proper subset of the columns of $\mu$ such that their size is divisible by $\mu_1$.  However, any subset of the columns of $\lambda$ is divisible by $\mu_1$.  Therefore, $(\lambda,\mu)$ is irreducible in $\Kostka_r$ and hence is in the Hilbert basis.  
\end{proof}

\begin{exmp}\label{exmp: gcd1}
Let $\lambda_1 = 20$ and $\mu_1 = 15$.  Since $\gcd(20,15) = 5$, $\gcd(21,15) = 3$, and $\gcd(22,15) = 1$, we have $r(\lambda_1,\mu_1) = 22$.  

The construction in the proof of \autoref{thm:: gcd1} yields the Hilbert basis element $(\lambda,\mu) \in \Kostka_{22}$, where
$$ \lambda = (\underbrace{20,\dots,20}_{15}) \text{ and } \mu = (\underbrace{15,\dots,15}_{7},\underbrace{13,\dots,13}_{15})\,.$$
Thus the pair  $(20,15)$ is $22$-initial.
\end{exmp}

The second sufficient condition of \autoref{thm: initial main} follows immediately from \autoref{thm:: gcd1}, as in this case we have $r(\lambda_1,\mu_1) \leq \lambda_1 + 1$.  We can now construct another family of examples to account for the last case of \autoref{thm: initial main}.

\begin{thm}\label{thm:: gcd2}
    Suppose $\gcd(\lambda_1 + 1, \mu_1 + 1) = 1$ and $2\mu_1 > \lambda_1+ 1$.  Then the pair $(\lambda_1,\mu_1)$ is $(\lambda_1 + 1)$-initial.
\end{thm}
\begin{proof}
    Let 
    $$\lambda = (\underbrace{\lambda_1,\dots,\lambda_1}_{2\mu_1 - \lambda_1 + 1}, \underbrace{\lambda_1-1,\dots,\lambda_1-1}_{\lambda_1 - \mu_1})$$
    and 
    $$\mu = (\underbrace{\mu_1,\dots,\mu_1}_{\lambda_1 + 1})\,.$$
    It is straightforward to check that  $\lambda$ dominates $\mu$, so $(\lambda,\mu)$ is in $\Kostka_{\lambda_1 + 1}$.  Observe that all but one of the columns of $\lambda$ have size $\mu_1 + 1$, while the last column has  size $2\mu_1 - \lambda_1 + 1$.  The columns of $\mu$ all have size $\lambda_1 + 1$.  

    By \autoref{rem: reducible condition}, if $(\lambda,\mu)$ is reducible, then we can choose a proper subset of the columns of $\lambda$ with the same size as a proper subset of the columns of $\mu$. If such a choice exists, note that the complement of the chosen columns also satisfies this property.  Thus, we can choose a subset of the columns of $\lambda$ excluding the smallest column of size equal to some subset of columns of $\mu$.  Note that the total size of any collection of columns of $\mu$ is divisible by $\lambda_1 + 1$.  Since we assume $\mu_1 + 1$ is coprime to $\lambda_1 + 1$, then a collection of at most $\lambda_1 - 1$ columns of size $\mu_1 + 1$ will not be divisible by $\lambda_1 + 1$. Therefore, no such set of columns exist.  We can conclude $(\lambda,\mu)$ is irreducible and hence is in the Hilbert basis of $\Kostka_{\lambda_1 + 1}$.
\end{proof}

\begin{exmp}
As in \autoref{exmp: gcd1}, we consider $\lambda_1 = 20$ and $\mu_1 = 15$.  Since $21$ and $16$ are coprime, \autoref{thm:: gcd2} applies to the pair $(\lambda_1,\mu_1)$.  The construction in the proof yields the Hilbert basis element $(\lambda,\mu) \in \Kostka_{21}$, where
$$ \lambda = (\underbrace{20,\dots,20}_{11}, \underbrace{19,\dots,19}_{5}) \text{ and } \mu = (\underbrace{15,\dots,15}_{21})\,.$$
Thus the pair  $(20,15)$ is $21$-initial, which is stronger than the statement yielded in \autoref{exmp: gcd1}.
\end{exmp}

Lastly, we show that the Hilbert basis elements we constructed lie on the $2$-skeleton of the Kostka cone by examining elements consisting of few distinct entries in $\Kostka_r$.

\begin{lem}\label{lem: two parts face}
If $\lambda,\mu \in \Par_r(n)$ are partitions satisfying $\lambda \geq_{\sf Dom}  \mu$ and that one is rectangular while the other has exactly two part sizes, then the point $(\lambda,\mu)$ lies on a $2$-dimensional face of $\Kostka_r$.
\end{lem}
\begin{proof}
    By \autoref{obs: kostka slice}, we can assume that the length of $\mu$ is $r$.  Suppose 
    $$\lambda = (\underbrace{x,\dots,x}_s) \text{ and }\mu = (\underbrace{y,\dots,y}_t,\underbrace{z,\dots,z}_{r-t})\,.$$
    Hence we have
    $$(\lambda,\mu) \in \left(\bigcap_{\substack{1 \leq i \leq r\\i \neq s}}H_i \right) \cap \left(\bigcap_{\substack{1 \leq j < r\\j \neq t}} \widehat H_j\right)\,.$$
    Thus the point $(\lambda,\mu)$ lies in the $2$-dimensional intersection of these $2r-3$ hyperplanes with the $(2r-1)$-dimensional cone $\Kostka_r$, and hence is a $2$-face of $\Kostka_r$.

    An analogous argument shows that if $\lambda$ is rectangular and $\mu$ has $k$ part sizes, then $(\lambda,\mu)$ lies on a $k$-dimensional face of $\Kostka_r$.
\end{proof}

Since the Hilbert basis elements we constructed satisfy the hypotheses of \autoref{lem: two parts face}, we can conclude the following.

\begin{cor}\label{cor: constructions 2d}
The $(\lambda,\mu)$ constructed in the proofs of \autoref{thm:: gcd1} and \autoref{thm:: gcd2} lie on a two-dimensional face of their respective Kostka cones.
\end{cor}

We can now combine these results to prove the main result.

\begin{proof}[Proof of \autoref{thm: initial main}]
The first sufficient condition follows from the Width Bound of Gao-Kiers-Orelowitz-Yong (\autoref{thm: width bound}) and the fact that if a pair is $r$-initial, then it is $r'$-initial for any $r' > r$.  The second and third sufficient conditions follow from \autoref{thm:: gcd1} and \autoref{thm:: gcd2}, respectively.  The final claim is a result of \autoref{cor: constructions 2d} and the fact that the Hilbert basis elements in \autoref{thm: width bound} are primitive vectors of extremal rays.
\end{proof}

\section{Further Directions}\label{sec: further directions}
We start by discussing a curious phenomenon in the $h$-vector of the $r$-Kostka polytope, namely, that half of the entries appear to take the value $1$. The \emph{$h$-vector} $(h_0,h_1,\dots,h_{d})$ of a $d$-polytope is defined from the \emph{$f$-vector} $(f_{-1},f_0,\dots,f_{d-1})$, where $f_k$ is the number of $k$-faces, by 
$$h_k = \sum_{i=0}^k (-1)^{k-i} \binom{d-i}{k-i} f_{i-1}\,.$$  
While $h$-vectors are usually studied in the case that the polytope is simple (or, dually, simplicial), recent work of Gaetz has shown that they can still have nice positivity properties in certain non-simple cases \cite{gaetz2023one}.  Though $\Kpoly_r$ is not simple and its $h$-vector can have negative entries, half of its $h$-vector still seems well-behaved.

\begin{conj}
    Let $(h_0,h_1,\dots,h_{2r-2})$ be the $h$-vector of $\Kpoly_r$.  Then $h_k = 1$ whenever $r - 1 \leq k \leq 2r - 2$.
\end{conj}

We have verified that the conjecture holds for all $r \leq 7$.  The only other instance we know of this phenomenon was observed by Charles Wang \cite{wangnumber} in studying the unordered partition polytope, which is the convex hull of the points $x \in \Z_{\geq 0}^n$ such that $(1,2,\dots,n) \cdot x = n$.  The facets of these polytopes were previously studied by Shlyk \cite{SHLYK20051139}.  It turns out that each unordered partition polytope is combinatorially equivalent to a face of some Kostka polytope.  It would be interesting to have an explanation for this phenomenon in either family of polytopes.  See \cite[Chapter 2]{stanley1996CCA} or \cite[Chapter 8]{ziegler2012lectures} for more details on $f$- and $h$-vectors.

\begin{exmp}
The $h$-vectors of $\Kpoly_r$ for $2 \leq r \leq 7$ are given by 
$(1,1,1)$, $(1,3,1,1,1)$, $(1,8,-3,1,1,1,1)$, $(1,17,-15,5,1,1,1,1,1)$, $(1,31,-36,13,1,1,1,1,1,1,1)$, and \\$(1,51,-60,2,25,-7,1,1,1,1,1,1,1)$.
\end{exmp}

Another avenue for potential progress is furthering the understanding of the face numbers of the Kostka polytope.  As we determined in \autoref{sec: edge description} and \autoref{sec: enumeration}, the function $f_{d}(r)$ counting the number of $d$-faces of the $r$-Kostka polytope is a polynomial of degree $3d+3$.  A more extensive computer calculation would allow one to determine this function for $d > 3$ via \autoref{thm: face recursion formula}. We have also shown that $f_{d}(r)$ has a positive integer expansion in the basis $\left\{ \binom{r}{i}\right\}_{i \geq 1}$.  It may be possible to explicitly express some integer coefficients in this expansion for arbitrary $d$ using an analogue of our methods for calculating the top degree coefficient.

\appendix
\section*{Appendix: Initial Pair Probability Computation}\label{app: initial pair}
In this appendix, we calculate the probability that two integers $\mu_1 < \lambda_1$ satisfy at least one of the conditions of \autoref{thm: initial main}.  Fix $N \in \Z_{>0}$, $B \in \Z_{>0} \cup \{\infty\}$, and let $I$ be a subset of $\{1,2,3\}$.  We then define $d(N,B,I)$ to be the proportion of integer pairs $(m,n)$ with $1 \leq m < n \leq N$ satisfying the restriction that $E_i$ holds for all $i \in I$, where the conditions are: 
\begin{itemize}
\item[$E_1$:] $m$ and $n$ have no common prime factors less than $B$,
\item [$E_2$:] $m$ and $n+1$ have no common prime factors less than $B$,
\item [$E_3$:] $m+1$ and $n+1$ have no common prime factors less than $B$, and $2m \geq n$.
\end{itemize}
Note that the case when $B = \infty$ is when the respective integers are coprime. By inclusion-exclusion, the desired probability is given by
$$\lim_{N \to \infty}\sum_{\text{nonempty } I \subseteq \{1,2,3\}} (-1)^{|I|+1} d(N,\infty,I)\,.$$
It remains to calculate $\lim_{N \to \infty}  d(N,\infty,I)$ for each nonempty $I \subseteq \{1,2,3\}$.  First, assume $3 \notin I$.  The Chinese Remainder Theorem implies that, for fixed $B \in \Z_{>0}$, we have
$$\lim_{N \to \infty} d(N,B,I) = \prod_{\text{prime }p \leq B} \left(1 - \frac{|I|}{p^2}\right)\,.$$
We now only need to account for the probability that our pairs of integers of interest are divisible by a large prime $p > B$.  By summing the probabilities for all such $p$, we see that the error $d(N,\infty,I) - d(N,B,I)$ vanishes as $B$ goes to infinity, since
$$\lim_{N \to \infty} d(N,B,I) - d(N,\infty,I) \leq \sum_{p > B} \frac{|I|}{p^2} \leq \int_B^\infty  \frac{|I|}{x^2} dx = \frac{|I|}{B}\,.$$
We can then conclude that 
$$\lim_{N \to \infty} d(N,\infty,I)  = \lim_{B \to \infty} \lim_{N \to \infty} d(N,B,I) =  \prod_{\text{prime }p} \left(1 - \frac{|I|}{p^2}\right)$$

For $k = 1,2,3$, let $\alpha_k = \prod_{\text{prime }p} \left( 1  - \frac{k}{p^2}\right)$.  A similar computation can be carried out when $3 \in I$, i.e, when we require $2m \geq n$ in addition to the divisibility properties, and the resulting probability is then $\alpha_{|I|}/2$.  The quantities $\alpha_k$ for $k = 1,2,3$ have decimal expansions described by the OEIS sequences \href{https://oeis.org/A059956}{A059956}, \href{https://oeis.org/A065474}{A065474}, and \href{https://oeis.org/A206256}{A206256}, respectively \cite{oeis}. We can then conclude that the desired probability is 
$$\frac{5}{2}\alpha_1 - 2\alpha_2 + \frac{1}{2}\alpha_3 \approx 0.93729304\,.$$

\addtocontents{toc}{\protect\setcounter{tocdepth}{0}}

\addtocontents{toc}{\protect\setcounter{tocdepth}{2}}

\nocite{*}
\bibliographystyle{amsplain}
\bibliography{bibliography.bib} 
\end{document}